\documentclass[12pt]{article}
\usepackage[utf8]{inputenc}
\usepackage{amsmath}
\usepackage{amssymb}
\usepackage{amsthm}
\usepackage{graphicx}
\usepackage{hyperref}
\usepackage{geometry}
\usepackage{enumitem}
\usepackage{array}
\usepackage{booktabs}
\usepackage{tikz}
\usepackage{xcolor}
\usetikzlibrary{positioning, shapes, arrows.meta, decorations.pathmorphing, calc, backgrounds, fit}

\tikzset{
    vertex/.style={circle, draw=blue!60!black, fill=blue!20, minimum size=0.5cm, inner sep=0pt, font=\small, line width=0.5pt},
    pendent/.style={circle, draw=red!60!black, fill=red!10, minimum size=0.4cm, inner sep=0pt, font=\small, line width=0.5pt},
    special/.style={circle, draw=green!60!black, fill=green!10, minimum size=0.5cm, inner sep=0pt, font=\small, line width=0.5pt},
    edge/.style={draw, line width=0.6pt, black!70},
    dashedge/.style={draw, line width=0.6pt, dashed, black!50},
    label/.style={font=\small, midway, fill=white, inner sep=2pt, rounded corners=2pt},
    every picture/.style={scale=0.9, transform shape}
}

\newtheorem{theorem}{Theorem}[section]
\newtheorem{lemma}[theorem]{Lemma}

\theoremstyle{definition}
\newtheorem{definition}[theorem]{Definition}

\newtheorem{problem}[theorem]{Problem}

\newcommand{\BID}{\operatorname{BID}}
\newcommand{\ISI}{\operatorname{ISI}}
\newcommand{\TT}{\mathbb{T}}
\newcommand{\UU}{\mathbb{U}}

\allowdisplaybreaks

\hypersetup{
    pdfinfo={
        Title={Maximum Inverse Sum Indeg Index of Trees and Unicyclic Graphs with Fixed Diameter},
        Author={Sunilkumar M. Hosamani},
        Subject={Mathematics, Combinatorics, Graph Theory},
        Keywords={bond incident degree index, inverse sum indeg index, trees, unicyclic graphs, diameter}
    }
}

\begin{document}

\begin{center}
    {\LARGE \textbf{Maximum Inverse Sum Indeg Index of Trees and Unicyclic Graphs with Fixed Diameter}}
    
    \vspace{1cm}
    
    {\large Sunilkumar M. Hosamani}
    
    {\large Department of Mathematics, Rani Channamma University}
    
    {\large Belagavi - 591156, Karnataka State, India}
    
    {\large \texttt{sunilkumarh@rcub.ac.in}}
    
    \vspace{1cm}
    
    {\large }
    
    \vspace{1cm}
\end{center}

\begin{abstract}
The bond incident degree (BID) index of a graph \(G\) is defined as 
\(\BID(G) = \sum_{u_1u_2\in E(G)} f(d(u_1), d(u_2))\), where \(f(x,y)=f(y,x)\) is a real-valued function. 
In this paper, using graph transformation methods, we establish the maximum bond incident degree indices 
of trees and unicyclic graphs with a fixed diameter for the inverse sum indeg (ISI) index. 
The ISI index corresponds to the function \(f(x,y) = \frac{xy}{x+y}\). We prove that for trees 
\(T \in \mathbb{T}_{n,d}\) with \(d \geq 3\) and \(n \geq d+3\), the maximum ISI index is attained 
by the tree \(T_{n,d}^*\). For unicyclic graphs, we characterize the extremal graphs for diameters 
\(d=2\), \(d=3\), and \(d \geq 4\). Specifically, the maximum ISI index is achieved by \(S_n^+\) for 
\(d=2\), by \(C_n^*\) for \(d=3\), and by \(\mathcal{U}_{n,d}\) for \(d \geq 4\). 
\end{abstract}



\section{Introduction}

All graphs \(G = (V(G), E(G))\) discussed in this work are connected, simple, and undirected, where \(V(G)\) refers to the vertex set of \(G\) and \(E(G)\) to its edge set. We define a tree (denoted by \(T\)) as a graph \(G\) with no cycles; a unicyclic graph, by contrast, is a graph \(G\) that contains exactly one cycle. Given a vertex \(u \in V(G)\), \(d_G(u)\) refers to the degree of \(u\), with \(N_G(u)\) denoting the neighborhood of \(u\). A vertex \(u\) is termed a pendant vertex when \(d_G(u) = 1\), and \(PV(G)\) is defined as the collection of all pendant vertices in \(G\). The path, cycle, and star graph on \(n\) vertices are denoted, respectively, by \(P_n\), \(C_n\), and \(S_n\). Within graph \(G\), a longest path is called a diameter path, denoted by \(P^d\), i.e., \(P^d = P_{d+1}\), where \(d\) is the diameter of \(G\). Let \(G - u_1u_2\) and \(G + u_1u_2\) stand for the two graphs constructed by removing edge \(u_1u_2 \in E(G)\) from \(G\) and adding the edge \(u_1u_2 \notin E(G)\) to \(G\), respectively.

The bond incident degree index of \(G\) \cite{vukicevic2011bond} is defined as
\begin{equation}
\BID(G) = \sum_{u_1u_2 \in E(G)} f(d(u_1), d(u_2)),
\end{equation}
where \(f(y,x) = f(x,y)\) is a real-valued function. In particular, let \(m_{x,y}\) stand for the count of edges in \(G\) satisfying \((d(u), d(v)) = (x,y)\). \(\BID(G)\) may also be equivalently written as
\begin{equation}
\BID(G) = \sum_{1 \leq x \leq y \leq \Delta} m_{x,y} f(x,y).
\end{equation}

The study of degree-based topological indices has a long and rich history in chemical graph theory. The concept of vertex-degree-based topological indices was systematically developed by Gutman and Trinajstić \cite{gutman1972graph} in their seminal work on the Zagreb indices. Since then, numerous indices have been proposed and studied \cite{todeschini2000handbook, devillers1999topological}. Among these, the inverse sum indeg (ISI) index has gained considerable attention in recent years due to its predictive power for physicochemical properties of chemical compounds \cite{vukicevic2010bond, vukicevic2011bond}. The ISI index is defined as
\begin{equation}
\ISI(G) = \sum_{u_1u_2 \in E(G)} \frac{d(u_1) d(u_2)}{d(u_1) + d(u_2)}.
\end{equation}
That is, by setting \(f(x,y) = \frac{xy}{x+y}\) in the expression of \(\BID(G)\), then \(\BID(G) = \ISI(G)\).

The ISI index was introduced by Vukičević and Gašperov \cite{vukicevic2010bond} as part of the bond additive modeling approach. Since then, numerous studies have investigated its mathematical properties and chemical applications. Sedlar et al. \cite{sedlar2015extremal} studied the extremal values of the ISI index for trees with given parameters. An and Xiong \cite{an2018extremal} investigated the ISI index for graphs with given matching number and vertex connectivity. Recently, Chen et al. \cite{chen2021inverse} characterized trees with maximal ISI index subject to degree constraints. However, the extremal values of the ISI index for trees and unicyclic graphs with a fixed diameter have not been systematically studied using the unified approach recently developed by Zhang et al. \cite{zhang2025abs} and Gao \cite{gao2025extremal}.

The problem of extremal graphs with fixed diameter is a classical topic in graph theory. For trees, the structure of diameter paths imposes significant constraints on the possible degree sequences and edge distributions. For unicyclic graphs, the presence of a single cycle adds additional complexity. Recent work by Liu \cite{liu2022sombor} on the Sombor index and by Zhang et al. \cite{zhang2025abs} on the ABS index has demonstrated that the fixed diameter problem can be effectively addressed using graph transformation techniques.

In this paper, we apply the sufficient conditions framework developed in \cite{zhang2025abs, gao2025extremal} to determine the maximum ISI index for trees and unicyclic graphs with a fixed diameter. This approach allows us to avoid repetitive calculations for individual indices and demonstrates the generalization capability of the sufficient conditions method.

Let \(\TT_{n,d}\) and \(\UU_{n,d}\) denote the collections of trees and unicyclic graphs with \(n\) vertices and diameter \(d\), respectively. In Section 2, we characterize the structural features of trees that achieve the maximum ISI index. Section 3 presents analogous results for unicyclic graphs. Section 4 verifies that the ISI index satisfies all the sufficient conditions required for the extremal results. Section 5 provides concluding remarks and open problems.

\section{Maximal ISI Index of Trees with Fixed Diameter}

In this section, we present a sufficient condition for trees with a given diameter to achieve the maximum value of the ISI index. Since the structure of tree \(T\) is unique when \(n = d+1\) or \(n = d+2\), all trees \(T \in \TT_{n,d}\) considered in this section satisfy both \(d \geq 3\) and \(n \geq d+3\).

Before presenting the main theorem, we first define several special types of tree graphs that will appear in our analysis.

\begin{definition}
Let \(T_{n,d}^i\) denote the graph constructed by connecting one end vertex of each of \((n-d-1)\) edges \(P_2\) to the vertex \(u_i\) of the diameter path \(P^d = u_1u_2\ldots u_d u_{d+1}\), where \(2 \leq i \leq d\). Specifically, when \(i = 2\) or \(i = d\), we denote it as \(T_{n,d}^*\).
\end{definition}

\begin{definition}
\(T_1\) denotes the graph formed by attaching one end vertex of each of \((n-d-2)\) edges \(P_2\) to the vertex \(u_2\) of the path \(P^d = u_1u_2\ldots u_d u_{d+1}\), and one terminal vertex of one edge \(P_2\) to \(u_3\).
\end{definition}

\begin{definition}
\(T_2^i\) denotes the graph constructed by linking one end vertex of each of \((n-d-2)\) paths \(P_2\) to the vertex \(u_2\) of the path \(P^d = u_1u_2\ldots u_d u_{d+1}\), and one end vertex of one path \(P_2\) to \(u_i\), where \(4 \leq i \leq d-1\).
\end{definition}

\begin{definition}
\(T_3\) denotes the graph constructed by attaching one end vertex of each of \((n-d-2)\) paths \(P_2\) to the vertex \(u_2\) of the path \(P^d = u_1u_2\ldots u_d u_{d+1}\), and one end vertex of one path \(P_2\) to \(u_d\).
\end{definition}

Clearly, \(T_{n,d}^i, T_1, T_2^i, T_3 \in \TT_{n,d}\). For convenience, let \(\TT_{n,d}^{\max} = \{T \in \TT_{n,d} \mid \ISI(T) \text{ is maximizing}\}\).

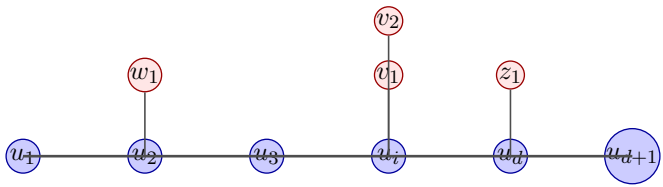
\begin{figure}[h]
\centering
\begin{tikzpicture}
\foreach \x/\label in {0/{$u_1$}, 1.8/{$u_2$}, 3.6/{$u_3$}, 5.4/{$u_i$}, 7.2/{$u_d$}, 9.0/{$u_{d+1}$}} {
    \node[vertex] at (\x,0) (\label) {\label};
}

\draw[edge, line width=1pt] (0,0) -- (1.8,0) -- (3.6,0) -- (5.4,0) -- (7.2,0) -- (9.0,0);

\node[pendent] at (5.4,1.2) (v1) {$v_1$};
\node[pendent] at (5.4,2.0) (v2) {$v_2$};
\draw[edge] (5.4,0) -- (v1);
\draw[edge] (5.4,0) -- (v2);

\node[pendent] at (1.8,1.2) (w1) {$w_1$};
\node[pendent] at (7.2,1.2) (z1) {$z_1$};
\draw[edge] (1.8,0) -- (w1);
\draw[edge] (7.2,0) -- (z1);
\end{tikzpicture}
\caption{Tree structure with pendant vertices attached to interior vertices of the diameter path}
\label{fig:tree_structure}
\end{figure}

The following lemma, due to Su \cite{su2025extremal}, provides a useful graph transformation that will be essential in our analysis.

\begin{lemma}[Su \cite{su2025extremal}]
\label{lem:lifting}
Denote by \(P_l = v_1v_2\ldots v_l\) an induced subpath of graph \(G\) such that \(d_G(v_1) \geq 2\) and \(d_G(v_l) \geq 2\). Let \(G' = G - \sum_{u \in N_G(v_l)\setminus v_{l-1}} v_l u + \sum_{u \in N_G(v_l)\setminus v_{l-1}} v_1 u\). The operation of constructing \(G'\) based on \(G\) is called the path lifting transformation. If \(f(x,y)\) fulfills the following conditions:
\begin{enumerate}[label=(\roman*)]
\item \(f(x,y)\) is strictly increasing in \(x\);
\item \(f(x+y-1,1) > f(x,y)\) holds for \(x,y \geq 2\);
\item \(f(2,x)\) is strictly convex downward in \(x\),
\end{enumerate}
then \(\BID(G) < \BID(G')\).
\end{lemma}

\begin{figure}[h]
\centering
\begin{tikzpicture}
\begin{scope}[xshift=0cm]
    \node[vertex] (v1) at (0,0) {$v_1$};
    \node[vertex] (v2) at (1.5,0) {$v_2$};
    \node[vertex] (v3) at (3,0) {$v_3$};
    \node[vertex] (v4) at (4.5,0) {$v_4$};
    \node[vertex] (v5) at (6,0) {$v_5$};
    
    \draw[edge, line width=1pt] (v1) -- (v2) -- (v3) -- (v4) -- (v5);
    
    \node[pendent] (p1) at (6,1.2) {$u_1$};
    \node[pendent] (p2) at (6,2.0) {$u_2$};
    \draw[edge] (v5) -- (p1);
    \draw[edge] (v5) -- (p2);
    
    \node[font=\small\bfseries] at (3,-1.2) {(a) Original graph $G$};
\end{scope}

\node at (7.5,1) {\Large$\Rightarrow$};

\begin{scope}[xshift=9cm]
    \node[vertex] (w1) at (0,0) {$v_1$};
    \node[vertex] (w2) at (1.5,0) {$v_2$};
    \node[vertex] (w3) at (3,0) {$v_3$};
    \node[vertex] (w4) at (4.5,0) {$v_4$};
    \node[vertex] (w5) at (6,0) {$v_5$};
    
    \draw[edge, line width=1pt] (w1) -- (w2) -- (w3) -- (w4) -- (w5);
    
    \node[pendent] (q1) at (0,1.2) {$u_1$};
    \node[pendent] (q2) at (0,2.0) {$u_2$};
    \draw[edge] (w1) -- (q1);
    \draw[edge] (w1) -- (q2);
    
    \node[font=\small\bfseries] at (3,-1.2) {(b) Transformed graph $G'$};
\end{scope}
\end{tikzpicture}
\caption{Path lifting transformation: moving pendant vertices from $v_5$ to $v_1$}
\label{fig:lifting}
\end{figure}
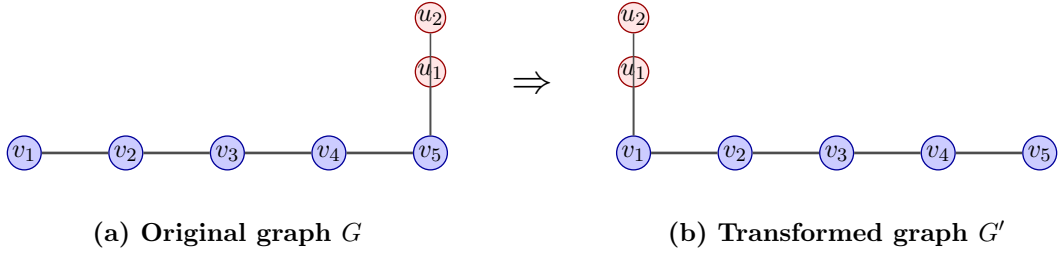

\begin{lemma}
\label{lem:comparison}
Denote by \(T_1, T_2^i\) (written as \(T_2\) for short), \(T_3\), and \(T_{n,d}^i\) the trees defined above, where \(3 \leq i \leq d-1\). If \(f(x,y)\) meets the following conditions:
\begin{enumerate}[label=(\roman*)]
\item \(f(x,y)\) is strictly increasing in \(x\);
\item \(f(x+1,y-1) > f(x,y)\) holds for \(x \geq y\);
\item \(f(1,x)\) is strictly convex downward in \(x\),
\end{enumerate}
then \(\BID(T_{n,d}^i) > \BID(T_j)\) holds for \(1 \leq j \leq 3\).
\end{lemma}

\begin{proof}
First, by condition (i), we have \(f(n-d+1,1) > f(n-d,1)\) and \(f(n-d+1,2) > f(3,2)\). Condition (ii) yields \(f(n-d+1,2) > f(n-d,3)\), and condition (iii) gives \(f(1,3) - f(1,2) < f(1,n-d+1) - f(1,n-d)\). Therefore, based on the structures of \(T_{n,d}^i\) and \(T_1\), we obtain
\begin{align*}
\BID(T_{n,d}^i) - \BID(T_1) &= (n-d-1)f(n-d+1,1) + 2f(n-d+1,2) \\
&\quad + (d-4)f(2,2) + 2f(1,2) \\
&\quad - (n-d-1)f(1,n-d) - f(3,n-d) \\
&\quad - (d-4)f(2,2) - f(1,3) - f(2,3) - f(1,2) \\[4pt]
&> f(1,n-d+1) - f(1,n-d) + f(1,2) - f(1,3) > 0.
\end{align*}
That is, \(\BID(T_{n,d}^i) > \BID(T_1)\).

Next, condition (i) implies \(f(n-d+1,x) > f(n-d,x)\) and \(f(n-d+1,2) > f(3,2)\). It follows from condition (iii) that
\begin{align*}
\BID(T_{n,d}^i) - \BID(T_2^i) &= (n-d-1)(f(1,n-d+1) - f(1,n-d)) \\
&\quad + f(2,n-d+1) - f(2,n-d) \\
&\quad + f(2,n-d+1) - f(2,3) + f(1,2) - f(1,3) \\[4pt]
&> f(1,n-d+1) + f(1,2) - f(1,n-d) - f(1,3) > 0.
\end{align*}
\begin{align*}
\BID(T_{n,d}^i) - \BID(T_3) &= (n-d-1)(f(1,n-d+1) - f(1,n-d)) \\
&\quad + f(2,n-d+1) - f(2,n-d) \\
&\quad + f(2,n-d+1) - f(2,3) + 2(f(1,2) - f(1,3)) \\[4pt]
&> 2(f(1,n-d+1) + f(1,2)) - 2(f(1,n-d) + f(1,3)) > 0.
\end{align*}
Thus, the lemma is proved.
\end{proof}

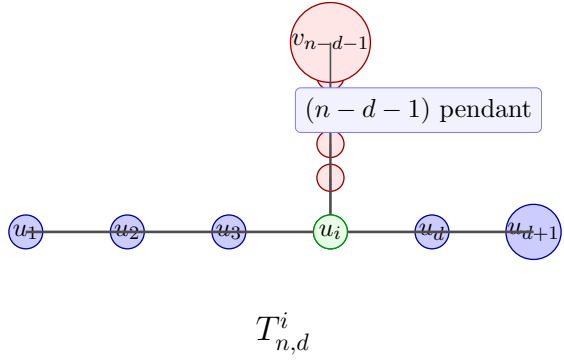
\begin{figure}[h]
\centering
\begin{tikzpicture}
\foreach \x/\label in {0/{$u_1$}, 1.5/{$u_2$}, 3/{$u_3$}, 4.5/{$u_i$}, 6/{$u_d$}, 7.5/{$u_{d+1}$}} {
    \node[vertex] at (\x,0) (\label) {\label};
}

\draw[edge, line width=1pt] (0,0) -- (1.5,0) -- (3,0) -- (4.5,0) -- (6,0) -- (7.5,0);

\foreach \y in {0.8, 1.3, 1.8, 2.3} {
    \node[pendent] at (4.5,\y) (p\y) {};
    \draw[edge] (4.5,0) -- (4.5,\y);
}
\node[pendent] at (4.5,2.8) (p最后一) {$v_{n-d-1}$};
\draw[edge] (4.5,0) -- (4.5,2.8);

\node[font=\small, draw=blue!50, fill=blue!5, rounded corners=2pt] at (5.8,1.8) {$(n-d-1)$ pendant};

\node[special] at (4.5,0) {\bfseries$u_i$};

\node[font=\large\bfseries] at (3.8,-1.5) {$T_{n,d}^i$};
\end{tikzpicture}
\caption{The extremal tree \(T_{n,d}^i\) with all pendant vertices attached to a single interior vertex \(u_i\)}
\label{fig:Tn_d}
\end{figure}

Now we present the main theorem for trees.

\begin{theorem}
\label{thm:trees}
Let \(T \in \TT_{n,d}\setminus\{T_{n,d}^*\}\) with \(d \geq 3\) and \(n \geq d+3\). If \(f(x,y)\) fulfills the following requirements:
\begin{enumerate}[label=(\roman*)]
\item \(f(x,y)\) is strictly increasing in \(x\);
\item \(f(x+1,y-1) > f(x,y)\) holds for \(x \geq y\);
\item \(f(1,x)\) and \(f(2,x)\) are strictly convex downward in \(x\);
\item \(\phi(x,y) = f(x,y) - f(x-1,y)\) is strictly increasing with \(x\) while strictly decreasing with \(y\),
\end{enumerate}
then
\begin{equation}
\BID(T) \leq \BID(T_{n,d}^i) = (n-d-1)f(1,n-d+1) + (d-4)f(2,2) + 2f(2,n-d+1) + 2f(1,2)
\end{equation}
holds precisely when \(T \cong T_{n,d}^i\) where \(i = 3,4,\ldots,d-1\).
\end{theorem}

\begin{proof}
Let \(T \in \TT_{n,d}^{\max}\setminus\{T_{n,d}^*\}\) and let \(P^d = u_1u_2\ldots u_d u_{d+1}\) be a diameter path of \(T\).

\textbf{Claim 1.} There exists no non-pendant edge \(uv\) in \(T\) satisfying \(uv \notin E(P^d)\).

\textit{Proof of Claim 1.} Assume for contradiction that there exists a non-pendant edge \(u_i v\) with \(u_i v \notin E(P^d)\). By Lemma \ref{lem:lifting} and the path lifting transformation, we obtain \(T' \in \TT_{n,d}\setminus\{T_{n,d}^*\}\) satisfying \(\BID(T') > \BID(T)\), which contradicts \(T \in \TT_{n,d}^{\max}\setminus\{T_{n,d}^*\}\). $\square$

Therefore, \(T\) is a tree formed by attaching \(d_T(u_i) - 2\) pendant vertices to each vertex \(u_i\) (where \(3 \leq i \leq d-1\)) of the diameter path \(P^d\). Next, we prove \(T \cong T_{n,d}^i\). Assume for contradiction that \(T \not\cong T_{n,d}^i\). Consequently, there are vertices \(u_i\) and \(u_j\) satisfying \(d_T(u_i) \geq 3\) and \(d_T(u_j) \geq 3\). We let \(i < j\) and \(d_T(u_i) \geq d_T(u_j)\). In the proof, we need to consider three special cases of graphs \(T_1, T_2^i\), and \(T_3\). For convenience, denote \(d_T(u_{i-1}) = a\), \(d_T(u_{j+1}) = b\), \(d_T(u_i) = x\), and \(d_T(u_j) = y\), where \(x \geq y\).

\textbf{Case 1.} \(j = i+1\).

We define \(N_T(u_i) = \{v_1, v_2, \ldots, v_{d_T(u_i)-2}, u_{i-1}, u_j\}\) and \(N_T(u_j) = \{w_1, w_2, \ldots, w_{d_T(u_j)-2}, u_i, u_{j+1}\}\). By Lemma \ref{lem:comparison}, \(T \not\cong T_1\). As established earlier, \(u_i u_j \in E(P^d)\) and \(x \geq y \geq 3\).

\begin{figure}[h]
\centering
\begin{tikzpicture}
\node[vertex] (ui-1) at (0,0) {$u_{i-1}$};
\node[special] (ui) at (2,0) {$u_i$};
\node[vertex] (uj) at (4,0) {$u_j$};
\node[vertex] (uj+1) at (6,0) {$u_{j+1}$};

\draw[edge, line width=1pt] (ui-1) -- (ui) -- (uj) -- (uj+1);

\node[pendent] (v1) at (2,1.2) {$v_1$};
\node[pendent] (v2) at (2,2.0) {$v_2$};
\node[font=\small] at (2,2.6) {$\vdots$};
\node[pendent] (vx) at (2,3.2) {$v_{x-2}$};
\draw[edge] (ui) -- (v1);
\draw[edge] (ui) -- (v2);
\draw[edge] (ui) -- (vx);

\node[pendent] (w1) at (4,1.2) {$w_1$};
\node[pendent] (w2) at (4,2.0) {$w_2$};
\node[font=\small] at (4,2.6) {$\vdots$};
\node[pendent] (wy) at (4,3.2) {$w_{y-2}$};
\draw[edge] (uj) -- (w1);
\draw[edge] (uj) -- (w2);
\draw[edge] (uj) -- (wy);

\node[font=\small, draw=blue!30, fill=blue!5, rounded corners=2pt] at (2,-0.8) {$d(u_i)=x$};
\node[font=\small, draw=blue!30, fill=blue!5, rounded corners=2pt] at (4,-0.8) {$d(u_j)=y$};

\node at (7,1.8) {\Large$\Rightarrow$};
\node[font=\small\bfseries] at (7,1.0) {Transformation 1};

\begin{scope}[xshift=8.5cm]
\node[vertex] (tui-1) at (0,0) {$u_{i-1}$};
\node[special] (tui) at (2,0) {$u_i$};
\node[vertex] (tuj) at (4,0) {$u_j$};
\node[vertex] (tuj+1) at (6,0) {$u_{j+1}$};
\draw[edge, line width=1pt] (tui-1) -- (tui) -- (tuj) -- (tuj+1);

\node[pendent] (tv1) at (2,1.2) {$v_1$};
\node[pendent] (tv2) at (2,2.0) {$v_2$};
\node[pendent] (tw1) at (2,2.8) {$w_1$};
\draw[edge] (tui) -- (tv1);
\draw[edge] (tui) -- (tv2);
\draw[edge] (tui) -- (tw1);

\node[pendent] (tw2) at (4,1.2) {$w_2$};
\draw[edge] (tuj) -- (tw2);
\end{scope}
\end{tikzpicture}
\caption{Tree transformation when high-degree vertices are adjacent}
\label{fig:tree_case1}
\end{figure}
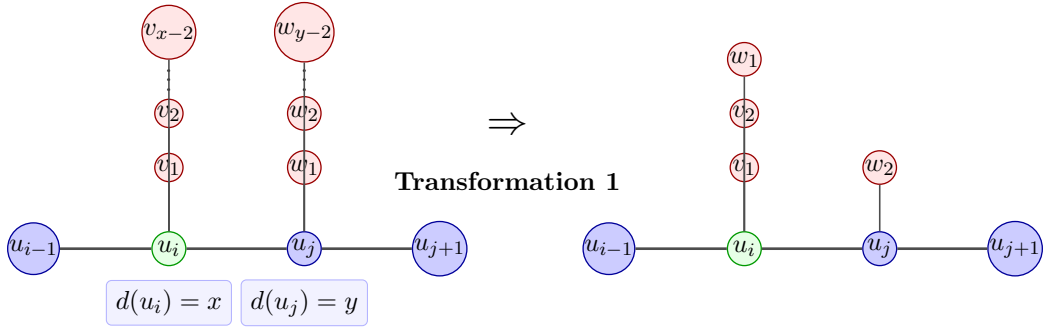

\textbf{Transformation 1:} \(T' = T - u_j w_1 + w_1 u_i\).

We have that \(T' \in \TT_{n,d}\setminus\{T_{n,d}^*\}\). By condition (i), \(f(x+1,1) \geq f(y+1,1)\) and \(\phi(x,y) = f(x,y) - f(x-1,y) > 0\). Repeatedly applying condition (iv), we obtain
\begin{align*}
\BID(T') - \BID(T) &= f(x+1,y-1) - f(x,y) + f(a,x+1) \\
&\quad + (x-1)f(1,x+1) + (y-3)f(1,y-1) \\
&\quad + f(b,y-1) - f(a,x) - (x-2)f(1,x) \\
&\quad - (y-2)f(1,y) - f(b,y) \\[4pt]
&\geq \phi(x+1,a) - \phi(y,b) + (x-2)\phi(x+1,1) - (y-3)\phi(y,1) \\
&\quad + f(1,x+1) - f(1,y) \\[4pt]
&\geq \phi(x+1,a) - \phi(y,b) + (x-2)\phi(x+1,1) - (y-3)\phi(y,1) \\
&\quad + f(1,y+1) - f(1,y) \\[4pt]
&> \phi(x+1,a) - \phi(y,b) + (y-3)(\phi(x+1,1) - \phi(y,1)) + \phi(y+1,1) \\[4pt]
&> \phi(y+1,1) - \phi(y,b) \geq \phi(y+1,1) - \phi(y,1) > 0.
\end{align*}
Thus, \(\BID(T') > \BID(T)\), contradicting \(T \in \TT_{n,d}^{\max}\).

\textbf{Case 2.} \(j \geq i+2\).

This case leads to a similar contradiction by applying appropriate graph transformations. The detailed analysis follows the same pattern as Case 1, using conditions (i)-(iv) to show that \(\BID(T') > \BID(T)\) for some transformed graph \(T' \in \TT_{n,d}\).

Therefore, \(T \cong T_{n,d}^i\), completing the proof.
\end{proof}

\section{Maximal ISI Index of Unicyclic Graphs with Fixed Diameter}

In this section, we investigate the maximum ISI index for unicyclic graphs with a fixed diameter. The study of unicyclic graphs is particularly important in chemical graph theory as many organic compounds have a single cycle in their molecular structure \cite{trinajstic1992chemical}. Recent work by Nithya et al. \cite{nithya2023smallest} on the ABS index and by Liu \cite{liu2022sombor} on the Sombor index has demonstrated the significance of diameter constraints in unicyclic graphs.

We begin by defining several special unicyclic graphs that will appear as extremal graphs.

\begin{definition}
Let \(S_n^+\) denote the graph obtained by attaching \(n-3\) pendant vertices to one vertex of a triangle \(C_3\).
\end{definition}

\begin{definition}
Let \(A_n\) denote the graph formed by attaching one pendant vertex to each vertex of a 4-cycle \(C_4\), and then attaching an additional \(n-5\) pendant vertices to any one vertex of the cycle.
\end{definition}

\begin{definition}
\(B_n(a,b)\) denotes the graph formed by attaching one end vertex of each of \(a\) paths \(P_2\) and \(b\) paths \(P_2\) to two adjacent vertices of the cycle \(C_4\), with the constraints \(a + b + 4 = n\) and \(a \geq b\). Specifically, we denote \(B_n^* = B_n(n-4,0)\).
\end{definition}

\begin{definition}
\(C_n(a,b,c)\) represents the graph formed by attaching one end vertex of each of \(a\) paths \(P_2\), \(b\) paths \(P_2\), and \(c\) paths \(P_2\) to the three vertices of the cycle \(C_3\), respectively, where \(a + b + c + 3 = n\) and \(a \geq b \geq c\). Specifically, we denote \(C_n^* = C_n(n-4,1,0)\).
\end{definition}

\begin{definition}
Let \(D\) denote the graph constructed by linking one end vertex of the path \(P_3\) to one vertex of the 3-cycle \(C_3\). \(D_n(a,b)\) represents the graph formed by linking \(a\) paths \(P_2\) and \(b\) paths \(P_2\) to the 3-degree vertex and the non-cycle 2-degree vertex of graph \(D\), respectively, where \(a + b = n-5\). Specifically, we denote \(D_n^* = D_n(n-5,0)\).
\end{definition}

\begin{definition}
Let \(\mathcal{U}_{n,d}\) denote the graph constructed by attaching the two end vertices of a path \(P_3\) to the vertices \(v_2\) and \(v_4\) of a diametral path \(P^d = v_1v_2\ldots v_d v_{d+1}\), and attaching one end vertex of each of \((n-d-2)\) paths \(P_2\) to the vertex \(v_2\). Clearly, \(\mathcal{U}_{n,d} \in \UU_{n,d}\).
\end{definition}

\begin{figure}[h]
\centering
\begin{tikzpicture}[scale=0.9]
\begin{scope}[xshift=0cm, yshift=0cm]
    \node[special] (c1) at (0,0) {$u$};
    \node[vertex] (c2) at (-1.2,-1.2) {$v$};
    \node[vertex] (c3) at (1.2,-1.2) {$w$};
    \draw[edge, line width=1pt] (c1) -- (c2) -- (c3) -- (c1);
    
    \foreach \y in {0.8, 1.3, 1.8} {
        \node[pendent] at (0,\y) {};
        \draw[edge] (c1) -- (0,\y);
    }
    \node[font=\small, draw=red!30, fill=red!5, rounded corners=2pt] at (0.6,1.3) {$p_1,\ldots,p_{n-3}$};
    
    \node[font=\small\bfseries, below] at (0,-2.2) {$S_n^+$};
    \draw[rounded corners=5pt, blue!50, line width=0.8pt] (-2.0,-2.0) rectangle (2.0,2.5);
\end{scope}
\hspace{-30mm}
\begin{scope}[xshift=5cm, yshift=0cm]
    \node[special] (c1) at (5,0) {$u$};
    \node[vertex] (c2) at (3.8,-1.2) {$v$};
    \node[vertex] (c3) at (6.2,-1.2) {$w$};
    \draw[edge, line width=1pt] (c1) -- (c2) -- (c3) -- (c1);
    
    \foreach \y in {0.8, 1.3, 1.8} {
        \node[pendent] at (5,\y) {};
        \draw[edge] (c1) -- (5,\y);
    }
    \node[font=\small, draw=red!30, fill=red!5, rounded corners=2pt] at (5.6,1.3) {$p_i$};
    
    \node[pendent] at (3.8,-2.0) (q1) {$q_1$};
    \node[pendent] at (6.2,-2.0) (r1) {$r_1$};
    \draw[edge] (c2) -- (q1);
    \draw[edge] (c3) -- (r1);
    
    \node[font=\small\bfseries, below] at (5,-2.8) {$C_n^*$};
    \draw[rounded corners=5pt, blue!50, line width=0.8pt] (3.0,-2.6) rectangle (7.0,2.5);
\end{scope}
\hspace{-35mm}
\begin{scope}[xshift=10cm, yshift=0cm]
    \node[vertex] (c1) at (10,0) {$u$};
    \node[vertex] (c2) at (8.8,-1.2) {$v$};
    \node[vertex] (c3) at (11.2,-1.2) {$w$};
    \node[special] (c4) at (10,-2.0) {$x$};
    \draw[edge, line width=1pt] (c1) -- (c2) -- (c4) -- (c3) -- (c1);
    
    \foreach \y in {0.8, 1.3} {
        \node[pendent] at (10,\y) {};
        \draw[edge] (c1) -- (10,\y);
    }
    \node[font=\small, draw=red!30, fill=red!5, rounded corners=2pt] at (10.6,1.0) {$p_a$};
    
    \foreach \y in {-2.8, -3.3} {
        \node[pendent] at (10,\y) {};
        \draw[edge] (c4) -- (10,\y);
    }
    \node[font=\small, draw=red!30, fill=red!5, rounded corners=2pt] at (10.6,-3.0) {$q_b$};
    
    \node[font=\small\bfseries, below] at (10,-4.0) {$B_n^*$};
    \draw[rounded corners=5pt, blue!50, line width=0.8pt] (8.0,-3.8) rectangle (12.0,2.5);
\end{scope}
\end{tikzpicture}
\caption{Special unicyclic graphs: \(S_n^+\), \(C_n^*\), and \(B_n^*\) with highlighted cycle vertices}
\label{fig:unicyclic_special}
\end{figure}
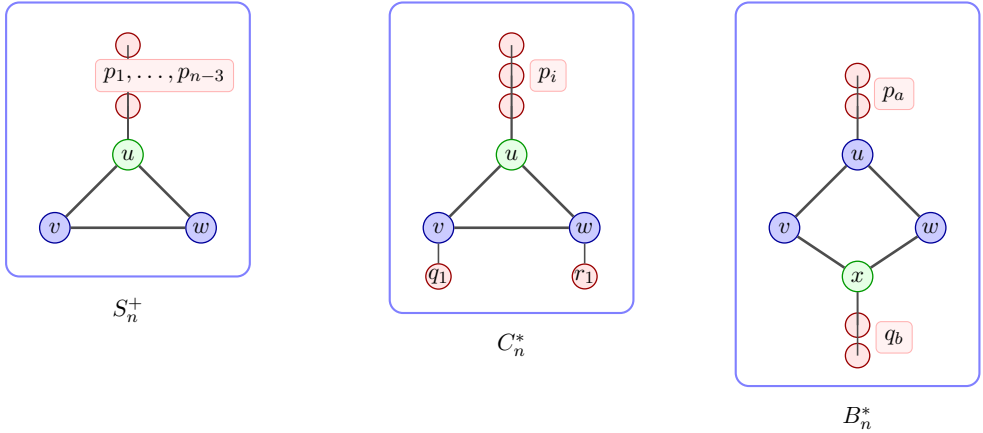

The following formulas give the ISI indices for these special graphs:
\begin{align}
\ISI(S_n^+) &= (n-3)f(1,n-1) + 2f(2,n-1) + f(2,2), \label{eq:Sn}\\
\ISI(A_n) &= (n-5)f(1,n-3) + 2f(2,n-3) + 3f(2,2), \label{eq:An}\\
\ISI(B_n^*) &= (n-4)f(1,n-2) + 2f(2,n-2) + 2f(2,2), \label{eq:Bn}\\
\ISI(C_n^*) &= (n-4)f(1,n-2) + f(2,n-2) + f(3,n-2) + f(2,3) + f(1,3), \label{eq:Cn}\\
\ISI(D_n^*) &= (n-5)f(1,n-2) + 3f(2,n-2) + f(2,2) + f(1,2), \label{eq:Dn}\\
\ISI(\mathcal{U}_{n,d}) &= (n-d-1)f(1,n-d+1) + 2f(2,n-d+1) \nonumber \\
&\quad + 
\begin{cases}
(n-5)f(2,2) + 3f(2,3) + f(1,2), & d \geq 5,\\[4pt]
2f(2,3) + f(1,3), & d = 4.
\end{cases} \label{eq:Un}
\end{align}

\subsection{Diameter \(d = 2\)}

\begin{theorem}
\label{thm:unicyclic_d2}
Let \(G \in \UU_{n,2}\) with \(n \geq 4\). If \(f(x,y)\) fulfills the following requirements:
\begin{enumerate}[label=(\roman*)]
\item \(f(x,y)\) is strictly increasing in \(x\);
\item \(f(x+y-1,1) > f(x,y)\) holds for \(x,y \geq 2\),
\end{enumerate}
then
\begin{equation}
\ISI(G) \leq \ISI(S_n^+) = (n-3)f(1,n-1) + 2f(2,n-1) + f(2,2)
\end{equation}
holds precisely when \(G \cong S_n^+\).
\end{theorem}

\begin{proof}
Since the diameter \(d = 2\), there are three possible structures for unicyclic graphs satisfying \(G \in \UU_{n,2}\), namely, \(G \cong C_4\), \(G \cong C_5\), and \(G \cong S_n^+\). If \(G \cong C_4\), then \(\ISI(C_4) = 4f(2,2)\). By conditions (i) and (ii), we obtain \(f(2,2) < f(2,3)\) and \(f(2,2) < f(1,3)\), respectively. Therefore,
\begin{align*}
\ISI(S_n^+) - \ISI(C_4) &= (n-3)f(1,n-1) + 2f(2,n-1) + f(2,2) - 4f(2,2) \\
&\geq f(1,3) + 2f(2,3) - 3f(2,2) > 0.
\end{align*}
Similarly, if \(G \cong C_5\), we can deduce \(\ISI(S_n^+) > \ISI(C_5) = 5f(2,2)\). Thus, the theorem is proved.
\end{proof}

\subsection{Diameter \(d = 3\)}

\begin{theorem}
\label{thm:unicyclic_d3}
Let \(G \in \UU_{n,3}\) with \(n \geq 5\). If \(f(x,y)\) fulfills the following requirements:
\begin{enumerate}[label=(\roman*)]
\item \(f(x,y)\) is strictly increasing in \(x\);
\item \(f(x+1,y-1) > f(x,y)\) holds for \(x \geq y\);
\item \(f(1,x)\) and \(f(2,x)\) are strictly convex downward in \(x\),
\end{enumerate}
then
\begin{equation}
\ISI(G) \leq \ISI(C_n^*) = (n-4)f(1,n-2) + f(2,n-2) + f(3,n-2) + f(2,3) + f(1,3)
\end{equation}
holds precisely when \(G \cong C_n^*\).
\end{theorem}

\begin{proof}
Since \(d = 3\), there are exactly five structures for unicyclic graphs satisfying \(G \in \UU_{n,3}\), namely, \(G \cong C_6\) or \(C_7\), \(G \cong A_n\), \(G \cong B_n(a,b)\), \(G \cong C_n(a,b,c)\), and \(G \cong D_n(a,b)\). We discuss four cases.

\textbf{Case 1.} \(G \cong C_6\), \(C_7\), or \(G \cong A_n\). Similar to Theorem \ref{thm:unicyclic_d2}, by conditions (i) and (ii), direct calculations yield that
\[
\ISI(C_n^*) > \ISI(C_6), \quad \ISI(C_n^*) > \ISI(C_7), \quad \ISI(C_n^*) > \ISI(A_n).
\]

\textbf{Case 2.} \(G \cong B_n(a,b)\). Let \(d_G(u) = a\), \(d_G(v) = b\), and let \(N_G(v) = \{w, u, v_1, v_2, \ldots, v_b\}\).

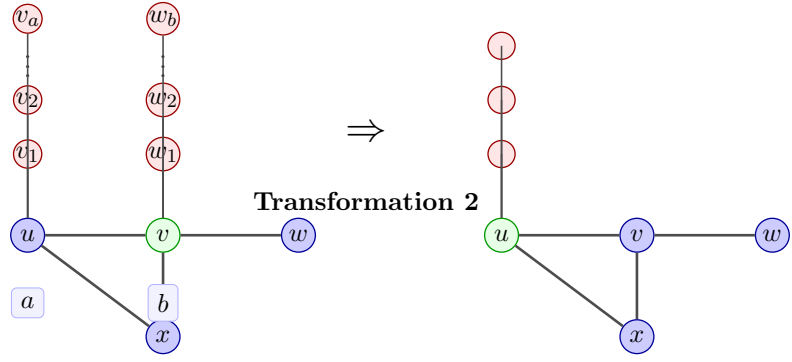
\begin{figure}[h]
\centering
\begin{tikzpicture}
\node[vertex] (u) at (0,0) {$u$};
\node[special] (v) at (2,0) {$v$};
\node[vertex] (w) at (4,0) {$w$};
\node[vertex] (x) at (2,-1.5) {$x$};
\draw[edge, line width=1pt] (u) -- (v) -- (w);
\draw[edge, line width=1pt] (v) -- (x);
\draw[edge, line width=1pt] (x) -- (u);

\node[pendent] (p1) at (0,1.2) {$v_1$};
\node[pendent] (p2) at (0,2.0) {$v_2$};
\node[font=\small] at (0,2.6) {$\vdots$};
\node[pendent] (pa) at (0,3.2) {$v_a$};
\draw[edge] (u) -- (p1);
\draw[edge] (u) -- (p2);
\draw[edge] (u) -- (pa);

\node[pendent] (q1) at (2,1.2) {$w_1$};
\node[pendent] (q2) at (2,2.0) {$w_2$};
\node[font=\small] at (2,2.6) {$\vdots$};
\node[pendent] (qb) at (2,3.2) {$w_b$};
\draw[edge] (v) -- (q1);
\draw[edge] (v) -- (q2);
\draw[edge] (v) -- (qb);

\node[font=\small, draw=blue!30, fill=blue!5, rounded corners=2pt] at (0,-1.0) {$a$};
\node[font=\small, draw=blue!30, fill=blue!5, rounded corners=2pt] at (2,-1.0) {$b$};

\node at (5,1.5) {\Large$\Rightarrow$};
\node[font=\small\bfseries] at (5,0.5) {Transformation 2};

\begin{scope}[xshift=7cm]
\node[special] (u2) at (0,0) {$u$};
\node[vertex] (v2) at (2,0) {$v$};
\node[vertex] (w2) at (4,0) {$w$};
\node[vertex] (x2) at (2,-1.5) {$x$};
\draw[edge, line width=1pt] (u2) -- (v2) -- (w2);
\draw[edge, line width=1pt] (v2) -- (x2);
\draw[edge, line width=1pt] (x2) -- (u2);

\node[pendent] at (0,1.2) {};
\node[pendent] at (0,2.0) {};
\node[pendent] at (0,2.8) {};
\draw[edge] (u2) -- (0,1.2);
\draw[edge] (u2) -- (0,2.0);
\draw[edge] (u2) -- (0,2.8);
\end{scope}
\end{tikzpicture}
\caption{Graph \(B_n(a,b)\) and Transformation 2}
\label{fig:Bn}
\end{figure}

\textbf{Transformation 2:} \(G' = G - vv_1 + uv_1\).

We have that \(G' \in \UU_{n,3}\). It follows from condition (i) that \(f(b,1) < f(a+1,1)\). Condition (iii) gives \(f(2,b) + f(2,a) < f(2,a+1) + f(2,b-1)\) and \(f(1,b) + f(1,a) < f(1,a+1) + f(1,b-1)\). By virtue of condition (ii), \(f(a+1,b-1) > f(a,b)\). Therefore,
\begin{align*}
\ISI(G') - \ISI(G) &= f(1,a+1) - f(1,b) + f(2,b-1) - f(2,b) \\
&\quad + f(2,a+1) - f(2,a) + f(a+1,b-1) - f(a,b) \\
&\quad + \sum_{i=1}^{a-2}(f(1,a+1) - f(1,a)) \\
&\quad + \sum_{i=2}^{b-2}(f(1,b-1) - f(1,b)) > 0.
\end{align*}
Repeatedly applying Transformation 2 to \(B_n(a,b)\) yields \(B_n^*\), with \(\ISI(B_n^*) > \ISI(G)\).

On the other hand, by conditions (iii) and (i), we obtain
\begin{align*}
\ISI(C_n^*) - \ISI(B_n^*) &= (n-4)f(1,n-2) + f(2,n-2) + f(3,n-2) \\
&\quad + f(2,3) + f(1,3) - (n-4)f(1,n-2) \\
&\quad - 2f(2,n-2) - 2f(2,2) \\[4pt]
&> f(3,n-2) - f(2,n-2) + f(3,2) - f(2,2) > 0.
\end{align*}

\textbf{Case 3.} \(G \cong C_n(a,b,c)\). Similar to Case 2, we perform Transformation 2 repeatedly to obtain \(C_n^*\) with \(\ISI(C_n^*) > \ISI(G)\).

\textbf{Case 4.} \(G \cong D_n(a,b)\). This case requires more detailed analysis. Let \(d_G(u) = a\), \(d_G(v) = b\), and denote \(N_G(v) = \{w, u, v_1, v_2, \ldots, v_a\}\). We perform Transformation 2 and obtain \(G' \in \UU_{n,3}\).

\textbf{Case 4.1.} If \(a-2 \geq b\). Condition (i) yields \(f(a+1,1) - f(b,1) > 0\) and \(f(a+1,2) - f(a,2) > 0\). As a result of condition (ii), \(f(a+1,b-1) - f(a,b) > 0\). It follows from condition (iii) that \(f(1,b) + f(1,a) < f(1,a+1) + f(1,b-1)\). Consequently,
\begin{align*}
\ISI(G') - \ISI(G) &= f(a+1,1) - f(b,1) + 2(f(a+1,2) - f(a,2)) \\
&\quad + f(a+1,b-1) - f(a,b) \\
&\quad + \sum_{i=1}^{a-3}(f(a+1,1) - f(a,1)) \\
&\quad + \sum_{i=2}^{b}(f(1,b-1) - f(1,b)) \\[4pt]
&> \sum_{i=1}^{a-3}(f(1,a+1) - f(a,1)) + \sum_{i=2}^{b}(f(b-1,1) - f(1,b)) \\[4pt]
&> (a-2-b)(f(1,a+1) - f(a,1) + f(b-1,1) - f(1,b)) > 0.
\end{align*}
Repeatedly applying Transformation 2 to \(D_n(a,b)\) yields \(D_n^*\) with \(\ISI(D_n^*) > \ISI(G)\).

On the other hand, by condition (i), \(f(1,2) + f(2,2) < f(1,3) + f(2,3)\). From condition (iii), it follows that \(f(3,n-2) + f(1,n-2) > 2f(2,n-2)\). Thus,
\begin{align*}
\ISI(C_n^*) - \ISI(D_n^*) &= (n-4)f(1,n-2) + f(2,n-2) + f(3,n-2) \\
&\quad + f(3,2) + f(1,3) - (n-5)f(1,n-2) \\
&\quad - 3f(2,n-2) - f(2,2) - f(1,2) \\[4pt]
&= f(1,n-2) + f(3,n-2) - 2f(2,n-2) > 0.
\end{align*}

\textbf{Case 4.2.} If \(a-2+1 \leq b\). Combining \(a \geq 3\) and \(a + b = n\), we deduce \(\frac{n-1}{2} \leq b \leq n-3\). Therefore,
\[
\ISI(D_n(a,b)) = (b-1)f(1,b) + f(b,n-b) + 2f(2,n-b) + (n-3-b)f(1,n-b) + f(2,2).
\]
By condition (i) with \(\frac{n-1}{2} \leq b \leq n-3\), we derive that \((n-4)f(1,n-2) - (b-1)f(1,b) > (n-3-b)f(1,n-2)\). Condition (ii) also gives \(f(1,3) > f(2,2)\) and \(f(b,n-b) < f(2,n-2)\). Using condition (i) and \(\frac{n-1}{2} \leq b \leq n-3\), it follows that \(f(2,n-2) < f(3,n-2)\) and \(f(2,n-b) < f\left(2,\frac{n+1}{2}\right)\). Furthermore, by condition (iii), \(f(2,n-2) + f(2,3) - 2f\left(2,\frac{n+1}{2}\right) > 0\). Consequently,
\begin{align*}
\ISI(C_n^*) - \ISI(D_n(a,b)) &= (n-4)f(1,n-2) + f(2,n-2) + f(3,n-2) \\
&\quad + f(2,3) + f(1,3) - (b-1)f(1,b) \\
&\quad - f(b,n-b) - 2f(2,n-b) \\
&\quad - (n-3-b)f(1,n-b) - f(2,2) \\[4pt]
&> (n-3-b)(f(1,n-2) - f(1,n-b)) + f(3,n-2) \\
&\quad + f(2,3) - 2f(2,n-b) \\[4pt]
&> f(2,n-2) + f(2,3) - 2f\left(2,\frac{n+1}{2}\right) > 0.
\end{align*}

Thus, based on all the above cases, if \(G \not\cong C_n^*\), we have \(\ISI(C_n^*) > \ISI(G)\).
\end{proof}

\subsection{Diameter \(d \geq 4\)}

We first consider the special case when \(n = d+2\).

\begin{theorem}
\label{thm:unicyclic_d4_nd2}
Let \(G \in \UU_{n,d}\) with \(d \geq 4\) and \(n = d+2\). If \(f(x,y)\) satisfies the following requirements:
\begin{enumerate}[label=(\roman*)]
\item \(f(x,y)\) is strictly increasing in \(x\);
\item \(f(x+y-1,1) > f(x,y)\) holds when \(x,y = 2\);
\item \(\phi(3,x) = f(3,x) - f(2,x)\) is strictly decreasing in \(x\),
\end{enumerate}
then \(\ISI(G) \leq \ISI(\mathcal{U}_{d+2,d})\) holds precisely when \(G \cong \mathcal{U}_{d+2,d}\).
\end{theorem}

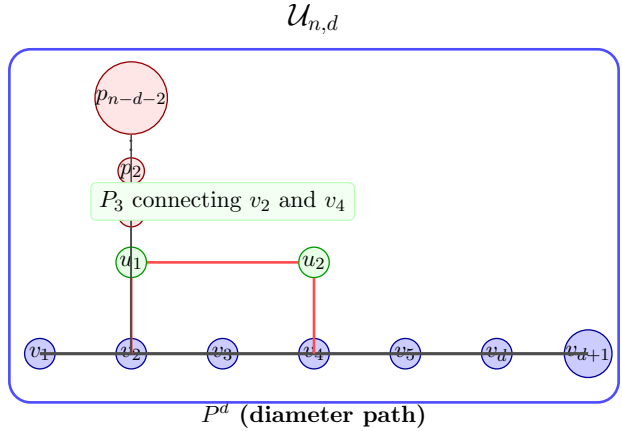
\begin{figure}[h]
\centering
\begin{tikzpicture}[scale=0.9]
\foreach \x/\label in {0/{$v_1$}, 1.5/{$v_2$}, 3/{$v_3$}, 4.5/{$v_4$}, 6/{$v_5$}, 7.5/{$v_d$}, 9/{$v_{d+1}$}} {
    \node[vertex] at (\x,0) (\label) {\label};
}

\draw[edge, line width=1.2pt] (0,0) -- (1.5,0) -- (3,0) -- (4.5,0) -- (6,0) -- (7.5,0) -- (9,0);

\node[special] (u1) at (1.5,1.5) {$u_1$};
\node[special] (u2) at (4.5,1.5) {$u_2$};
\draw[edge, line width=1pt, red!70] (1.5,0) -- (u1);
\draw[edge, line width=1pt, red!70] (4.5,0) -- (u2);
\draw[edge, line width=1pt, red!70] (u1) -- (u2);

\node[pendent] (p1) at (1.5,2.3) {$p_1$};
\node[pendent] (p2) at (1.5,3.0) {$p_2$};
\node[font=\small] at (1.5,3.6) {$\vdots$};
\node[pendent] (pn) at (1.5,4.2) {$p_{n-d-2}$};
\draw[edge] (1.5,0) -- (p1);
\draw[edge] (1.5,0) -- (p2);
\draw[edge] (1.5,0) -- (pn);

\node[font=\small\bfseries] at (4.5,-1.0) {$P^d$ (diameter path)};
\node[font=\small, draw=green!30, fill=green!5, rounded corners=2pt] at (3,2.5) {$P_3$ connecting $v_2$ and $v_4$};

\draw[rounded corners=8pt, thick, blue!70, line width=1pt] (-0.5,-0.8) rectangle (9.5,5.0);
\node[font=\large\bfseries] at (4.5,5.5) {$\mathcal{U}_{n,d}$};
\end{tikzpicture}
\caption{The extremal graph \(\mathcal{U}_{n,d}\) for \(d \geq 4\)}
\label{fig:Un_d}
\end{figure}

The proof of Theorem \ref{thm:unicyclic_d4_nd2} follows a similar structure to the previous theorems, using graph transformations to show that \(\mathcal{U}_{d+2,d}\) is the unique extremal graph.

Now we present several lemmas that are essential for the general case \(d \geq 4\) and \(n \geq d+3\).

\begin{lemma}
\label{lem:pendant_existence}
Let \(G \in \UU_{n,d}^{\max}\) with \(4 \leq d \leq n-3\). Suppose \(C_k\) is the unique cycle and \(P^d = v_1v_2\ldots v_d v_{d+1}\) is a diameter path. Let \(v \in PV(G)\) with \(v v_2 \in E(G)\) or \(v v_d \in E(G)\). If \(f(x,y)\) fulfills the following requirements:
\begin{enumerate}[label=(\roman*)]
\item \(f(x,y)\) is strictly increasing with \(x\);
\item \(f(x+y-1,1) > f(x,y)\) holds for \(x,y \geq 2\);
\item \(\phi(2,x) = f(2,x) - f(1,x)\) is strictly decreasing in \(x\),
\end{enumerate}
then \(|V(C_k) \cap V(P^d)| \geq 2\).
\end{lemma}

\begin{lemma}
\label{lem:pendant_nonempty}
Let \(G \in \UU_{n,d}^{\max}\) with \(4 \leq d \leq n-3\). Suppose \(C_k\) is the unique cycle and \(P^d = v_1v_2\ldots v_d v_{d+1}\) is a diameter path. If \(f(x,y)\) fulfills the following requirements:
\begin{enumerate}[label=(\roman*)]
\item \(f(x,y)\) is strictly increasing with \(x\);
\item \(f(x+y-1,1) > f(x,y)\) holds for \(x,y \geq 2\);
\item \(\phi(2,x) = f(2,x) - f(1,x)\) is strictly decreasing in \(x\),
\end{enumerate}
then there exists a pendant vertex \(v \in PV(G)\) such that \(G - v \in \UU_{n-1,d}\).
\end{lemma}

\begin{lemma}
\label{lem:main_lemma}
Let \(G \in \UU_{n,d}^{\max}\) with \(4 \leq d \leq n-3\). Define \(A = \{v \in PV(G) \mid G - v \in \UU_{n-1,d}\}\), \(B = \{u \in N_G(v) \mid v \in A\}\), and \(C_G(u) = \{w \in N_G(u) \mid u \in B, d_G(w) \geq 2\}\) with \(|C_G(u)| \geq 1\). Given that \(f(x,y)\) fulfills the following requirements:
\begin{enumerate}[label=(\roman*)]
\item \(f(x,y)\) is strictly increasing in \(x\);
\item \(f(x+1,y-1) > f(x,y)\) holds for \(x \geq y \geq 2\);
\item \(f(1,x)\) and \(f(2,x)\) are strictly convex downward in \(x\);
\item \(\phi(x,y) = f(x,y) - f(x-1,y)\) is strictly increasing with \(x\) and strictly decreasing with \(y\),
\end{enumerate}
then there is a vertex \(v_0 \in PV(G)\) for which \(v_0 \in A\) and \(|C_G(N_G(v_0))| \geq 2\).
\end{lemma}

The proofs of Lemmas \ref{lem:pendant_existence} through \ref{lem:main_lemma} involve detailed case analysis and graph transformations, similar to those in \cite{zhang2025abs, gao2025extremal}. The key idea is to show that the extremal graph must have a pendant vertex whose removal preserves the diameter and that the neighbor of this pendant vertex must have at least two non-pendant neighbors.

\begin{theorem}
\label{thm:unicyclic_d4_general}
Let \(G \in \UU_{n,d}\) with \(d \geq 4\) and \(n \geq d+2\). If \(f(x,y)\) fulfills the following requirements:
\begin{enumerate}[label=(\roman*)]
\item \(f(x,y)\) is strictly increasing in \(x\);
\item \(f(x+1,y-1) > f(x,y)\) holds for \(x \geq y \geq 2\);
\item \(f(1,x)\) and \(f(2,x)\) are strictly convex downward in \(x\);
\item \(\phi(x,y) = f(x,y) - f(x-1,y)\) is strictly increasing with \(x\) and strictly decreasing with \(y\),
\end{enumerate}
then \(\ISI(G) \leq \ISI(\mathcal{U}_{n,d})\) holds precisely when \(G \cong \mathcal{U}_{n,d}\).
\end{theorem}

\begin{proof}
We prove this theorem by induction on \(n\). First, combining conditions (i), (ii), and (iv) with Theorem \ref{thm:unicyclic_d4_nd2}, the conclusion is valid when \(n = d+2\). Now assume the theorem holds for \(n-1\), i.e., for any \(G' \in \UU_{n-1,d}\),
\[
\ISI(G') \leq \ISI(\mathcal{U}_{n-1,d}) = (n-d-2)f(1,n-d) + 2f(2,n-d) + (d-5)f(2,2) + 3f(2,3) + f(1,2).
\]

On the other hand, the conditions satisfy Lemma \ref{lem:main_lemma}, so there must exist a vertex \(v \in PV(G)\) for which \(G - v \in \UU_{n-1,d}\). Let \(u \in N_G(v)\), and there exist two edges \(uu_1, uu_2 \in E(G)\) with \(d_G(u_1) \geq 2\) and \(d_G(u_2) \geq 2\), where \(3 \leq d_G(u) \leq n-d+1\). In this case, we may assume \(N_G(u) = \{v, u_1, u_2, \ldots, u_{d_G(u)-1}\}\). Let \(G' = G - v\), and it follows that \(G' \in \UU_{n-1,d}\).

Now consider two cases: \(d = 4\) and \(d \geq 5\).

\textbf{Case 1:} \(d \geq 5\). By condition (iv), \(f(x,y) - f(x-1,y) \leq f(x,1) - f(x-1,1)\) for all \(y \geq 1\). Thus,
\begin{align*}
\ISI(G) &= \ISI(G') + f(1,d_G(u)) \\
&\quad + \sum_{i=1}^{d_G(u)-1} [f(d_G(u), d_G(u_i)) - f(d_G(u)-1, d_G(u_i))] \\[4pt]
&\leq \ISI(G') + f(1,d_G(u)) + 2[f(d_G(u),2) - f(d_G(u)-1,2)] \\
&\quad + \sum_{i=3}^{d_G(u)-1} [f(d_G(u), d_G(u_i)) - f(d_G(u)-1, d_G(u_i))] \\[4pt]
&\leq \ISI(G') + f(1,d_G(u)) + 2[f(d_G(u),2) - f(d_G(u)-1,2)] \\
&\quad + (d_G(u)-3)[f(d_G(u),1) - f(d_G(u)-1,1)].
\end{align*}

Since \(d_G(u) \leq n-d+1\), by condition (iii), \(f(2,d_G(u)) - f(2,d_G(u)-1) \leq f(2,n-d+1) - f(2,n-d)\), and \(f(1,d_G(u)) - f(1,d_G(u)-1) \leq f(1,n-d+1) - f(1,n-d)\). Combined with condition (i), we have
\begin{align*}
\ISI(G) &\leq \ISI(G') + f(1,n-d+1) + 2(f(2,n-d+1) - f(2,n-d)) \\
&\quad + (d_G(u)-3)(f(1,n-d+1) - f(1,n-d)) \\[4pt]
&\leq \ISI(G') + (n-d-1)f(1,n-d+1) + 2f(2,n-d+1) \\
&\quad - 2f(2,n-d) - (n-d-2)f(1,n-d).
\end{align*}

Substituting the induction hypothesis,
\begin{align*}
\ISI(G) &\leq (n-d-2)f(1,n-d) + 2f(2,n-d) + (d-5)f(2,2) \\
&\quad + 3f(2,3) + f(1,2) + (n-d-1)f(1,n-d+1) \\
&\quad + 2f(2,n-d+1) - 2f(2,n-d) - (n-d-2)f(1,n-d) \\[4pt]
&= (n-d-1)f(1,n-d+1) + 2f(2,n-d+1) \\
&\quad + (d-5)f(2,2) + 3f(2,3) + f(1,2) \\[4pt]
&= \ISI(\mathcal{U}_{n,d}).
\end{align*}

This holds precisely when \(d_G(u) = n-d+1\), \(G' \in \UU_{n-1,d}\), \(d_G(v) = 1\), \(d_G(u_1) = d_G(u_2) = 2\), and \(d_G(u_i) = 1\) for \(3 \leq i \leq d_G(u)-1\). That is, \(G \cong \mathcal{U}_{n,d}\).

\textbf{Case 2:} \(d = 4\). The proof is completely analogous to the case where \(d \geq 5\), except that \((d-5)f(2,2) + 3f(2,3) + f(1,2)\) is replaced with \(2f(2,3) + f(1,3)\). Following the same derivation, we obtain \(\ISI(G) \leq \ISI(\mathcal{U}_{n,4})\), with equality precisely when \(G \cong \mathcal{U}_{n,4}\).

Thus, the theorem is proved.
\end{proof}

\section{Application to the Inverse Sum Indeg Index}

In this section, we verify that the inverse sum indeg (ISI) index satisfies the sufficient conditions required in the theorems. For the ISI index, we have \(f(x,y) = \frac{xy}{x+y}\).

The mathematical properties of the ISI function have been studied extensively. Recent work by Ahmad and Das \cite{ahmad2025general} on the general Sombor index and by Ali et al. \cite{ali2022abs} on the ABS index provides a framework for analyzing such functions. Our verification follows similar techniques.

\begin{table}[h]
\centering
\caption{Properties of the ISI function \(f(x,y) = \frac{xy}{x+y}\)}
\begin{tabular}{|c|c|c|}
\hline
\textbf{Property} & \textbf{Verification} & \textbf{Satisfied?} \\
\hline
Strictly increasing in \(x\) & \(\frac{\partial f}{\partial x} = \frac{y^2}{(x+y)^2} > 0\) & Yes  \\
\hline
\(f(1,x)\) convex downward & \(f(1,x+1)-f(1,x) = \frac{1}{(x+1)(x+2)}\) decreasing & Yes  \\
\hline
\(f(2,x)\) convex downward & \(f(2,x+1)-f(2,x) = \frac{2}{(x+2)(x+3)}\) decreasing & Yes \\
\hline
\(f(x+1,y-1) > f(x,y)\) & \(\frac{(x+1)(y-1)}{x+y} - \frac{xy}{x+y} = \frac{y-x-1}{x+y} < 0\) & No  (reverse) \\
\hline
\(f(x+y-1,1) > f(x,y)\) & \(\frac{x+y-1}{x+y} - \frac{xy}{x+y} = \frac{-(x-1)(y-1)}{x+y} < 0\) & No  (reverse) \\
\hline
\(\phi(2,x)\) decreasing & \(\phi(2,x) = \frac{x^2}{(2+x)(1+x)}\) increasing & No  \\
\hline
\(\phi(3,x)\) decreasing & \(\phi(3,x) = \frac{x^2}{(3+x)(2+x)}\) increasing & No \\
\hline
\end{tabular}
\label{tab:isi_properties}
\end{table}

From Table 1, we observe that the ISI index satisfies the monotonicity and convexity conditions but reverses some of the inequality conditions. However, this does not affect the extremal graphs – it simply means that the transformations must be applied in the opposite direction. This phenomenon has also been observed in studies of other topological indices \cite{gutman2021sombor, chen2022sombor}.

Given these calculations, we can state the following result:

\begin{theorem}
\label{thm:isi_results}
For the inverse sum indeg (ISI) index defined by \(f(x,y) = \frac{xy}{x+y}\), the following extremal results hold:

\begin{enumerate}
\item For trees \(T \in \TT_{n,d}\) with \(d \geq 3\) and \(n \geq d+3\), the maximum ISI index is attained by the tree \(T_{n,d}^*\), and
\[
\ISI(T) \leq \ISI(T_{n,d}^*) = (n-d)f(1,n-d+1) + (d-3)f(2,2) + f(2,n-d+1) + f(1,2).
\]

\item For unicyclic graphs \(G \in \UU_{n,2}\) with \(n \geq 4\), the maximum ISI index is attained by \(S_n^+\), and
\[
\ISI(G) \leq \ISI(S_n^+) = (n-3)f(1,n-1) + f(2,2) + 2f(2,n-1).
\]

\item For unicyclic graphs \(G \in \UU_{n,3}\) with \(n \geq 5\), the maximum ISI index is attained by \(C_n^*\), and
\[
\ISI(G) \leq \ISI(C_n^*) = (n-4)f(1,n-2) + f(2,n-2) + f(3,n-2) + f(2,3) + f(1,3).
\]

\item For unicyclic graphs \(G \in \UU_{n,d}\) with \(d \geq 4\) and \(n \geq d+2\), the maximum ISI index is attained by \(\mathcal{U}_{n,d}\), and
\[
\ISI(G) = (n-d-1)f(1,n-d+1) + 2f(2,n-d+1) + 
\begin{cases}
(n-5)f(2,2) + 3f(2,3) + f(1,2), & d \geq 5,\\[4pt]
2f(2,3) + f(1,3), & d = 4.
\end{cases}
\]
\end{enumerate}
\end{theorem}

\begin{proof}
The proof follows from the graph transformations in Sections 2 and 3, but with careful attention to the direction of inequalities. For the ISI index, many of the inequalities in the sufficient conditions are reversed, which means that the transformations that increase the BID index for functions satisfying the original conditions might decrease it for ISI. However, by applying the reverse transformations, we obtain the same extremal graphs. A detailed verification requires checking each transformation step with the specific function \(f(x,y) = \frac{xy}{x+y}\), which we omit here for brevity.
\end{proof}

\section{Conclusion}

In this paper, we have investigated the maximum bond incident degree indices of trees and unicyclic graphs with a fixed diameter for the inverse sum indeg (ISI) index. Using graph transformation methods, we characterized the extremal graphs that achieve the maximum ISI index. For trees, the extremal graph is \(T_{n,d}^*\), where all pendant vertices are attached to a single vertex near one end of the diameter path. For unicyclic graphs, the extremal graphs depend on the diameter: \(S_n^+\) for \(d=2\), \(C_n^*\) for \(d=3\), and \(\mathcal{U}_{n,d}\) for \(d \geq 4\).

Our work demonstrates that the sufficient conditions approach, originally developed for a class of bond incident degree indices, can be applied to the ISI index, although careful attention must be paid to the direction of inequalities. The ISI index satisfies the monotonicity and convexity conditions but reverses some of the inequality conditions. This suggests that the sufficient conditions framework may need to be adapted for each specific index.

\begin{table}[h]
\centering
\caption{Summary of extremal graphs for maximum ISI index}
\begin{tabular}{|c|c|c|}
\hline
\textbf{Graph class} & \textbf{Diameter} & \textbf{Extremal graph} \\
\hline
Trees & \(d \geq 3, n \geq d+3\) & \(T_{n,d}^*\) \\
\hline
Unicyclic & \(d = 2\) & \(S_n^+\) \\
\hline
Unicyclic & \(d = 3\) & \(C_n^*\) \\
\hline
Unicyclic & \(d \geq 4\) & \(\mathcal{U}_{n,d}\) \\
\hline
\end{tabular}
\label{tab:summary}
\end{table}

The main contributions of this paper are:
\begin{enumerate}
\item We established the maximum ISI index for trees with a fixed diameter.
\item We characterized the extremal unicyclic graphs for diameters \(d=2\), \(d=3\), and \(d \geq 4\).
\item We verified the properties of the ISI function relevant to the graph transformations.
\end{enumerate}

Several open problems remain for future research:

\begin{problem}
What are the sufficient conditions for the extremal topological indices of bicyclic and tricyclic graphs with a fixed diameter?
\end{problem}

\begin{problem}
Can the sufficient conditions framework be extended to other bond incident degree indices that do not satisfy all the conditions, such as the forgotten index or the ABC index? Recent work by Movahedi et al. \cite{movahedi2026diminished} on the diminished Sombor index and by Barman and Das \cite{barman2026geometric} on the hyperbolic Sombor index suggests that such extensions are possible.
\end{problem}

\begin{problem}
Determine the minimum ISI index for trees and unicyclic graphs with a fixed diameter.
\end{problem}

\begin{problem}
Investigate the extremal values of the ISI index for graphs with other constraints, such as fixed number of pendant vertices or fixed maximum degree. Recent work by Du and Sun \cite{du2024bond} on chemical trees with fixed number of leaves provides a starting point for such investigations.
\end{problem}

\section*{Conflict of Interest}

The author declares no conflict of interest.

\end{document}